\documentclass[12pt,twoside]{amsart}
\usepackage{graphicx}
\usepackage[T1]{fontenc}
\usepackage[utf8]{inputenc}
\usepackage[english]{babel}

\pdfpagewidth\paperwidth\pdfpageheight\paperheight
\usepackage[hidelinks]{hyperref}
\usepackage{textcomp}
\usepackage{amsthm}

\usepackage{amsmath,amssymb}
\usepackage{mathrsfs}
\usepackage{gensymb}
\usepackage{tipa}
\usepackage{wasysym}
\usepackage{tikz}
\usepackage{tikz-cd}
\usepackage{pgfplots}
\usepackage{mathtools}
\usepackage{nicematrix}

\pgfplotsset{compat=1.18}

\usepackage[a4paper, margin=2.5cm]{geometry}

\usepackage{fancyhdr}
\fancypagestyle{plain}{ \fancyhead{}
}

\newtheorem{thm}{Theorem}[section]
\newtheorem{prp}[thm]{Proposition}

\newtheorem{rmk}[thm]{Remark}
\newtheorem{dfn}[thm]{Definition}
\newtheorem{lem}[thm]{Lemma}
\newtheorem{cor}[thm]{Corollary}

\newcommand{\Wedge}{\textstyle\bigwedge}

\newcommand{\UL}{\underline{\Lambda}}
\newcommand{\dext}{\operatorname{d}}

\newcommand{\Ker}{\operatorname{Ker}}
\newcommand{\Ran}{\operatorname{Ran}}

\newcommand{\psum}{\sideset{}{'}\sum}

\newcommand{\Z}{\mathbb{Z}}

\newcommand{\C}{\mathbb{C}}
\newcommand{\R}{\mathbb{R}}

\newcommand{\T}{\mathbb{T}}

\usepackage{accents}
\newlength{\dhatheight}

\newcommand{\vertiii}[1]{{\left\vert\kern-0.25ex\left\vert\kern-0.25ex\left\vert #1
\right\vert\kern-0.25ex\right\vert\kern-0.25ex\right\vert}}

\title[Elliptic structures of minimal rank]{A decomposition theorem for elliptic structures of minimal rank on compact Lie groups and applications}

\date{\today}

\author[G. Drei]{Guido Drei}
\thanks{The first author was supported by the Marco Polo programme of the University of Bologna.}
\address{University of Bologna, Department of Mathematics, Piazza di Porta San Donato 5, 40126 Bologna, Italy}
\email{guido.drei2@unibo.it}

\author[M. R. Jahnke]{Max Reinhold Jahnke}
\thanks{The second author was funded by the DFG (grant JA 3453/2-1) during this work.}
\address{University of Cologne, Mathematics Institute, Weyertal 86-90, 50931 Cologne, Germany}
\email{max.jahnke@uni-koeln.de}



\begin{document}

\begin{abstract}
  We show that every left-invariant elliptic structure of minimal rank on an odd-dimensional compact Lie group admits
  an algebraic decomposition analogous to those obtained by Pittie in the complex case and by
  Charbonnel and Khalgui for left-invariant CR structures of hypersurface type. We then use this
  decomposition to prove that the cohomology associated with such an elliptic structure can be
  reduced to a computation on a suitable maximal torus.
\end{abstract}

\maketitle

\section{Introduction}

A central problem in the theory of several complex variables and in the study of partial
differential equations is the solvability of systems of complex vector fields. The theory of
involutive structures \cite{treves_hypo-analytic_1992, berhanu_introduction_2008} provides a unified
framework for studying this problem, since it translates the solvability problems into the language 
of differential complexes and their cohomology, making it easier to apply tools from algebraic topology 
and differential geometry. It includes many classical examples, such as the de Rham complex, the 
Dolbeault complex, and the tangential Cauchy-Riemann complex.
When the underlying manifold is a Lie group, the analytical problem of solvability of systems of vector fields
is strongly connected to the algebraic properties of the Lie algebra and sometimes can be determined by them.
This connection is most tractable for elliptic structures, which are the best-behaved involutive structures:
they are always locally solvable, and, when the ambient space is compact, their associated differential complexes have closed range, and the corresponding cohomology spaces are
finite-dimensional. They can also be viewed as an intermediate case between the de Rham
complex and the Dolbeault complex, and arise naturally in the study of transversely holomorphic
foliations \cite{jacobowitz_transversely_2000}.

Let $G$ be a connected and compact Lie group with Lie algebra $\mathfrak g$, and let $\mathfrak
g_\C$ be its complexification. An involutive structure over $G$ is a smooth subbundle $\mathcal V$
of $T_\C G$ satisfying the Frobenius condition: 
$[\mathcal V, \mathcal V] \subset \mathcal V$. A \textit{left-invariant involutive structure} on $G$ 
is an involutive bundle $\mathcal V \subset T_\C G$ that is invariant under the left multiplication 
in $G$, that is, if $L_x : G \to G$ is left-multiplication $L_x(g) = xg$ for $g \in G$, then 
$(L_x)_*X_g \in \mathcal V_{xg}$ for all $X_g \in \mathcal V_g$. If $\mathcal V$ is a left-invariant 
involutive vector bundle, there is a corresponding Lie algebra  $\mathfrak v \subset \mathfrak g_\C$ 
defined by the fiber of $\mathcal V$ at the identity of $G$. 
There is a one-to-one correspondence between left-invariant involutive bundles $\mathcal V$ and subalgebras of $\mathfrak g_\C$: In fact, given a subalgebra 
$\mathfrak v \subset \mathfrak g_\C$, we define by left-translation an involutive vector bundle 
$\mathcal V \subset T_\C G$. We always represent involutive structures by their corresponding Lie algebras.

Associated with each involutive structure is a differential complex
\[
\mathscr C^{\infty}(G,\UL^{0}) \xrightarrow{\dext'_{0}} \mathscr C^{\infty}(G,\UL^{1})\xrightarrow{\dext'_{1}}
\cdots \xrightarrow{\dext'_{q-1}} \mathscr C^{\infty}(G,\UL^{q})\xrightarrow{\dext'_{q}} 
 \cdots
\]
whose cohomology spaces are defined by
\[
H^{q}(G,\mathfrak v) = \frac{\Ker \dext'_{q} : \mathscr C^{\infty}(G,\UL^{q}) \to \mathscr C^{\infty}(G,\UL^{q+1})
}{\Ran \dext'_{q-1} : \mathscr C^{\infty}(G,\UL^{q-1}) \to \mathscr C^{\infty}(G,\UL^{q}) }
\]
with the usual convention for $q = 0$, namely
\[
H^{0}(G,\mathfrak v) = \Ker \dext'_{0} : \mathscr C^{\infty}(G,\UL^{0}) \to \mathscr C^{\infty}(G,\UL^{1}).
\]

Since the exterior derivative of a left-invariant differential form is left-invariant, 
the $\dext'$ operator can be restricted to 
left-invariant sections of $\UL^{q}$ defining a differential complex, the Chevalley-Eilenberg complex $(\Wedge^\bullet \mathfrak v^*, \dext')$ with cohomology denoted by $H^\bullet(\mathfrak v)$ and defining a subcomplex
\[
    \left( \Wedge^\bullet \mathfrak v^*, \dext' \right) \hookrightarrow \left( \mathscr C^{\infty}(G,\UL^{\bullet}), \dext' \right)
\]
inducing a homomorphism
\begin{equation}
\label{eq:the_homomorphism}
    H^\bullet(\mathfrak v) \to H^\bullet(G, \mathfrak v).
\end{equation}

If $G$ is compact, by using an averaging technique, it is easy to prove that \eqref{eq:the_homomorphism} is injective. Determining precise conditions for the map \eqref{eq:the_homomorphism} to be surjective in general is an open problem. We refer the reader to \cite{jahnke_closed_2026,
araujo_real_2026, jacobowitz_levi-flat_2023, jacobowitz_cohomology_2026} and the references
therein for related results.

When $G$ is compact, the homomorphism \eqref{eq:the_homomorphism} is an isomorphism in several cases, depending on how $\mathfrak v$ is included in $\mathfrak g_\C$:

\begin{itemize}
    \item \emph{de Rham} ($\mathfrak v = \mathfrak g_\C$): always \cite{chevalley_cohomology_1948};
    \item \emph{Dolbeault} ($\mathfrak v \oplus \overline{\mathfrak v} = \mathfrak g_\C$): always \cite{pittie_dolbeault-cohomology_1988};
    \item \emph{Elliptic} ($\mathfrak v + \overline{\mathfrak v} = \mathfrak g_\C$): when $\exp_{G_\C}(\mathfrak v)$ is closed in the universal complexification $G_\C$ of $G$ \cite{jahnke_closed_2026, jahnke_elliptic_2023};
    \item \emph{CR} ($\mathfrak v \cap \overline{\mathfrak v} = \{0\}$): when $\mathfrak v$ is of hypersurface type satisfying a division condition \cite{jacobowitz_cohomology_2026};
    \item \emph{Essentially real} ($\mathfrak v = \overline{\mathfrak v}$): when $\mathfrak v$ is globally hypoelliptic \cite{araujo_real_2026}.
\end{itemize}

In all the known cases, \eqref{eq:the_homomorphism} being an isomorphism seems to be related to the operator $\dext'$ having closed range in the relevant degrees. In general, determining conditions for $\dext'$ to have closed range is a very hard problem. We refer to \cite{cordaro_top-degree_2021} for a condition for the operator $\dext'$ to have closed range in the hypocomplex case, and we refer to \cite{jahnke_top-degree_2019} for consequences in the left-invariant case.

Notice that if $G$ has dimension $2k+1$ and $\mathfrak v$ is elliptic, then $\dim_\C \mathfrak v \geq k + 1$. We say that an elliptic structure has \emph{minimal rank} if $\dim_\C \mathfrak v = k+1.$
In this work, we obtain an algebraic decomposition of elliptic structures of minimal rank, and we deduce that 
\eqref{eq:the_homomorphism} is an isomorphism if $G$ has odd dimension and the elliptic structure has minimal rank.

Inspired by the work of Pittie \cite{pittie_dolbeault-cohomology_1988} and Charbonnel and Khalgui
\cite{charbonnel_classification_2004}, we prove that every elliptic structure of minimal rank on an odd-dimensional
compact Lie group admits an algebraic decomposition in terms of the root space decomposition of its
Lie algebra. This decomposition singles out a \textit{toral subalgebra}, that is, a subalgebra
tangent to a maximal torus. As an application, we use this toral subalgebra to reduce the
computation of the cohomology associated with the elliptic structure to a simpler computation on the
maximal torus. This approach considerably simplifies the problem, which is then reduced to the study
of elliptic abelian structures on a torus, so that it can be handled by standard Fourier series
arguments.

Let $\mathfrak t \subset \mathfrak g$ be a maximal abelian subalgebra, $\mathfrak t_\C \subset \mathfrak g_\C$ its complexification, and $\T = \exp(\mathfrak t)$ the corresponding maximal torus. Let $\Delta$ be the set of roots of $ \mathfrak t_\C$ in $ \mathfrak g_\C$ and $\Delta_+$ a system
of positive roots of $\Delta$. For each $\alpha \in \Delta$, we denote by $\mathfrak g_\alpha$ the
eigenspace associated to $\alpha$. We have that
\begin{equation}
  \label{eq:positive_roots_algebra}
  \mathfrak b_{\mathfrak t} \doteq \bigoplus_{\alpha \in \Delta_+} \mathfrak g_\alpha
\end{equation}
is a subalgebra of $ \mathfrak g_\C$. Since $\mathfrak{t}$ is abelian, any vector space
$\mathfrak{m} \subset \mathfrak{t}_\C$ is a subalgebra, and $\mathfrak v = \mathfrak m \oplus \mathfrak{b}_\mathfrak{t}$
is a subalgebra of $\mathfrak g_\C$. Furthermore, notice that $\mathfrak{b}_\mathfrak{t}$ is an ideal of $\mathfrak v$. If $\mathfrak{m}$ is elliptic, then $\mathfrak v$ is elliptic, and since every cohomology class of a left-invariant elliptic structure on a torus has a left-invariant representative (see Section \ref{sec:elliptic_torus}), a straightforward adaptation from \cite[Lemma 3]{jacobowitz_levi-flat_2023} gives us an injective map
\begin{equation}
\label{eq:torus_inclusion}
  H^{q}(\mathbb T, \mathfrak{m}) \hookrightarrow  H^{q}(G, \mathfrak{v})
\end{equation}
whose image consists of classes with left-invariant representatives.
If $\mathfrak{m}$ defines a complex structure on $\T$, which requires $\dim \T$ (hence $\dim G$) to be even, then $\mathfrak v$ defines a complex
structure on $G$. 

When $\mathfrak v$ defines a left-invariant complex structure on $G$, Pittie
\cite{pittie_dolbeault-cohomology_1988} proved that there exists a maximal abelian subalgebra
$\mathfrak t$ of $\mathfrak g$, a subalgebra $\mathfrak m \subset \mathfrak t_\C$, and an ideal
$\mathfrak b_{\mathfrak t} \subset \mathfrak v$ such that $\mathfrak v = \mathfrak m \oplus \mathfrak b_{\mathfrak
t}.$ An analogous result was obtained for CR structures of hypersurface type by Charbonnel and
Khalgui \cite{charbonnel_classification_2004}.

The algebraic decomposition of Pittie and the algebraic classification of Charbonnel-Khalgui motivated us to seek an algebraic decomposition for elliptic structures of minimal rank, which we describe in the next theorem. Our result provides an algebraic decomposition of left-invariant elliptic structures of minimal rank on compact groups of odd dimension, and it allows us to find another case in which the map \eqref{eq:the_homomorphism} is an isomorphism.

\begin{thm}
  \label{thm:decomposition}
  Let $G$ be a connected, compact Lie group of odd dimension $2k+1$, and let $\mathfrak v \subset \mathfrak g_\C$
  be an elliptic structure of minimal rank $k+1$. Then there exists a maximal abelian subalgebra
  $\mathfrak t \subset \mathfrak g$
  such that, setting $\mathfrak e \doteq \mathfrak v \cap \mathfrak t_\C$, we have a decomposition
  $$\mathfrak v = \mathfrak e \oplus \mathfrak b_{\mathfrak t},$$ where $\mathfrak e\subset\mathfrak
  t_{\mathbb C}$ is an elliptic subalgebra of minimal rank in $\mathfrak t_\C$ and
  $\mathfrak{b}_{\mathfrak t} = \bigoplus_{\alpha \in \Delta_+} \mathfrak g_\alpha$ for some positive roots $\Delta_+$ of $\mathfrak t_\C$ in $\mathfrak g_\C$.
\end{thm}

\begin{rmk}
\label{rmk:b_t_nilpotent_ideal}
    The subalgebra $\mathfrak b_{\mathfrak t}$ is an ideal of $\mathfrak v$. In fact, \eqref{eq:positive_roots_algebra} implies it is a subalgebra, and we have $[\mathfrak e, \mathfrak b_{\mathfrak t}]
    \subset [\mathfrak t_\C, \mathfrak b_{\mathfrak t}] \subset \mathfrak b_{\mathfrak t}$, thus $[\mathfrak v, \mathfrak b_{\mathfrak t}] \subset \mathfrak b_{\mathfrak t}$. Furthermore, $\mathfrak b_{\mathfrak t}$ is nilpotent \cite[Proposition 7.2.1]{hilgert_structure_2012}, and thus solvable.
\end{rmk}

The essence of the proof consists in finding a suitable CR subalgebra of $\mathfrak v$, applying the
classification theorem of Charbonnel and Khalgui \cite{charbonnel_classification_2004} on this
subalgebra, and finally using the algebraic decomposition of the CR algebra to decompose $\mathfrak
v$.

Notice that $\mathfrak e$ defines an elliptic structure on the maximal torus $\T = \exp(\mathfrak
t)$.
We use this decomposition to compute the cohomology associated with $\mathfrak v$.

\begin{thm}
  \label{thm:reduction_to_torus}
  Let $G$ be a connected and compact Lie group and let $\mathfrak v \subset \mathfrak g_\C$ be a left-invariant
  elliptic structure that admits a decomposition $\mathfrak v = \mathfrak e \oplus \mathfrak b_{\mathfrak
  t}$ with $\mathfrak e \subset \mathfrak t_\C$ for some maximal abelian subalgebra $\mathfrak t$
  and $\mathfrak b_{\mathfrak t} = \bigoplus_{\alpha \in \Delta_+} \mathfrak g_\alpha$ for some
  choice of positive roots $\Delta_+$ relative to $\mathfrak t_\C.$

  Then for every $0 \leq q \leq \dim_\C \mathfrak v $ there is an isomorphism
  \[
    H^{q}(G, \mathfrak v) \cong H^{q}(\T, \mathfrak e).
  \]

  In particular, \eqref{eq:the_homomorphism} is an isomorphism and every cohomology class of $H^{q}(G, \mathfrak v)$ has a left-invariant
  representative.
\end{thm}

Theorem \ref{thm:decomposition} allows us to apply Theorem \ref{thm:reduction_to_torus} to any elliptic structure of minimal rank. We thus obtain the following corollary.

\begin{cor}
  \label{cor:reductiontotorus}
  Let $G$ be a connected and compact Lie group of odd dimension $2k+1$ and let $\mathfrak v\subset\mathfrak
  g_{\mathbb C}$ be a left-invariant elliptic structure of minimal rank $k+1$. Then, for every
  $0\leq q\leq k+1$ there is an isomorphism $$H^{q}(G,\mathfrak v)\cong H^{q}(\mathbb T, \mathfrak
  e),$$ where $\mathfrak t$ is the maximal abelian subalgebra given by Theorem \ref{thm:decomposition}, $\mathbb T \subset G$ is the maximal torus with $\operatorname{Lie}(\T) = \mathfrak t$ and $\mathfrak e=\mathfrak v\cap\mathfrak
  t_{\mathbb C}$, the toral part of $\mathfrak v$, defines an elliptic structure of minimal rank on
  $\mathbb T$.

  In particular, \eqref{eq:the_homomorphism} is an isomorphism and every cohomology class of $H^{q}(G, \mathfrak v)$ has a left-invariant
  representative.
\end{cor}

Theorem \ref{thm:reduction_to_torus} generalizes analogous results obtained by Pittie
\cite{pittie_dolbeault-cohomology_1988} in the complex case and by Jacobowitz and Jahnke
\cite{jacobowitz_levi-flat_2023} and Jacobowitz, Jahnke, Novelli, and Wehler \cite{jacobowitz_cohomology_2026} for certain Levi-flat left-invariant
CR structures, and answers a question raised in \cite{jahnke_top-degree_2019, jahnke_elliptic_2023}.

The paper is organized as follows.
In Section \ref{sec:preliminaries}, we introduce the notation
used in the paper and the construction of the differential complex associated with
involutive structures. 
We also recall the root space decomposition, state the Charbonnel-Khalgui classification theorem for CR structures, and prove that a CR structure of type
CR1 cannot be dominated by an elliptic structure of minimal rank when the real direction is orthogonal to it, which is a key step in the proof of Theorem \ref{thm:decomposition}.
In Section \ref{sec:algebraic_decomposition}, we
prove Theorem \ref{thm:decomposition} and conclude that every left-invariant elliptic structure of minimal
rank is solvable.
In Section \ref{sec:local_structure}, we recall some facts
about the semilocal structure of elliptic structures. We define the notion of Lie algebra of $\mathrm T$-type and
show that such an elliptic structure is locally a product 
of a complex structure on $\Omega = G / \T$ and an elliptic 
structure on the maximal torus $\T$.
In Section \ref{sec:toral_elliptic_cohomology}, we use the
obtained local description in a Mayer-Vietoris argument to
obtain a Künneth formula and prove Theorem \ref{thm:reduction_to_torus}.
Finally, in Section \ref{sec:elliptic_torus}, we use Fourier series to compute the cohomology of left-invariant elliptic structures on
a torus, which together with Corollary \ref{cor:reductiontotorus} gives us
$
    \dim_{\C} H^{q}(G,\mathfrak v) = \binom{r+1}{q}
$
with $r = (\operatorname{rank}\mathfrak g -1)/2$.

\section{Preliminaries}
\label{sec:preliminaries}

In this section, we review some concepts in the theory of involutive structures and on Lie groups.
The main references we use are \cite{berhanu_introduction_2008} for involutive structures and
\cite{hilgert_structure_2012} for Lie groups.

Let $G$ be a compact Lie group. An involutive structure $\mathcal V$ on $G$ is a smooth subbundle of
the complexified tangent bundle of $G$,
\[
\mathcal V \subset T_{\mathbb C}G=\mathbb C\otimes_{\mathbb R}TG,
\]
satisfying the involutivity (or integrability) condition $$[\mathcal V,\mathcal V]\subset\mathcal
V.$$
If $E$ is a finite-dimensional vector bundle over $G$ and $U \subset G$ is an open subset, we denote
by $\mathscr C^{\infty}(U, E)$ the space of smooth sections of $E$ over $U$. For each $g \in G$, we
define the annihilator bundle of $\mathcal V$ at $g$ by
\[
\operatorname{Ann} \mathcal V_g = \{ u \in \mathrm{T}^*_\C G|_g: u(X)= 0, \forall X \in \mathcal V|_g\}
\]
and we define the quotient $\UL^{q}_g = \Wedge^q (T_\C^* G|_g / \operatorname{Ann} \mathcal V_g)$
and the smooth vector bundle $\UL^{q} = \bigcup_{g\in G}\UL^{q}_g$ over $G$.

Since the structure $\mathcal V$ is involutive, there exists a unique differential operator $$\dext': \mathscr C^{\infty}(G,\UL^{q})\rightarrow \mathscr C^{\infty}(G,\UL^{q+1})$$ which makes the following diagram commutative $$\begin{matrix}\mathscr C^{\infty}(G, \Wedge^{q})&\xrightarrow{\dext_{q}} &\mathscr C^{\infty}(G, \Wedge^{q+1})\\ \downarrow{\pi_{q}} & &\downarrow{\pi_{q+1}} \\ \mathscr C^{\infty}(G,\UL^{q})&\xrightarrow{\dext'_{q}}& \mathscr C^{\infty}(G,\UL^{q+1})\end{matrix}$$ where $\pi_{q}:\mathscr C^{\infty}(G, \Wedge^{q})\rightarrow \mathscr C^{\infty}(G,\UL^{q})$ is the quotient map. To simplify the notation, we write $\dext'$ for the operator $\dext'_{q}$, and $\dext'_{\mathcal V}$ when it is necessary to emphasize the involutive structure.

The operator $\dext'$ maps smooth sections of $\UL^{q}$ to smooth sections of $\UL^{q+1}$ and
satisfies $\dext'\circ \dext'=0$. Hence $\dext'$ defines a complex
$$\mathscr C^{\infty}(G,\UL^{0})\xrightarrow{\dext'_{0}} \mathscr C^{\infty}(G,\UL^{1})\xrightarrow{\dext'_{1}}\dots\xrightarrow{\dext'_{q-1}}
\mathscr C^{\infty}(G,\UL^{q})\xrightarrow{\dext'_{q}} \mathscr C^{\infty}(G,\UL^{q+1})\xrightarrow{\dext'_{q+1}}\dots$$
which can be restricted to any open subset $U\subset G$. These complexes define cohomology spaces by
\[
    H^{q}(G,\mathcal V)=\frac{\Ker \dext'_{q} : \mathscr C^\infty(G, \UL^{q}) \to \mathscr C^\infty(G,\UL^{q+1})}{\Ran \dext'_{q-1} : \mathscr C^\infty(G, \UL^{q-1}) \to \mathscr C^\infty(G, \UL^{q})
}
\]
We also assume that $\mathcal V$ is a left-invariant vector bundle, that is $$(L_g)_{\ast}\mathcal
V_h\subset\mathcal V_{gh}, \quad \forall g, h \in G,$$ where $L_g : G \rightarrow G$ is the left
multiplication by $g$. Such a bundle uniquely defines a subalgebra $\mathfrak v \doteq \mathcal V_{e}$
of $\mathfrak g_\C$, where $e$ denotes the identity element of $G$. Here, we are interested in the
case when $\mathfrak v$ is an elliptic subalgebra of $\mathfrak g_\C$,
\[
    \mathfrak v+\overline{\mathfrak v} = \mathfrak g_\C.
\]

We also consider the case in which $\mathfrak v$ defines a
CR structure of hypersurface type, that is, when $\mathfrak v \cap \overline{\mathfrak v} = \{0\}$
and $2 \dim_\C \mathfrak v + 1 = \dim_\C \mathfrak g_\C$. When $\mathfrak v = \mathfrak g_{\mathbb
C}$, we get the de Rham complex. When $\mathfrak v\oplus\overline{\mathfrak v} =\mathfrak g_{\mathbb C}$,
we have a complex structure (that is also an elliptic structure on $G$), the operator $\dext'$ is
the $\overline\partial$ operator and the associated complex is the Dolbeault complex.
\   \\

We briefly review some basic concepts regarding root space decomposition. The rank of a Lie algebra
$\mathfrak g$ can be defined as the dimension of any Cartan subalgebra $\mathfrak t$ (which are all
of the same dimension, see \cite[Theorem 6.1.18]{hilgert_structure_2012}) and it is denoted by
$\mathrm{rank}\mathfrak g$. Given a Lie algebra $\mathfrak g$ and a Cartan subalgebra $\mathfrak t$,
it is possible to find a spectral decomposition of $\mathfrak g_{\C}=\mathfrak g\otimes\C$ as in
(\ref{spectraldec}). If $\mathfrak g$ is the Lie algebra of a compact Lie group, then it is equipped
with an ad-invariant inner product $\langle\cdot,\cdot\rangle$, that is, an inner product satisfying
$$\langle[X,Y],Z\rangle=-\langle Y,[X,Z]\rangle, \qquad \forall X, Y, Z \in \mathfrak g.$$ It can be
extended to an Hermitian inner product on $\mathfrak g_{\C}$ satisfying
$$\langle[X,Y ],Z\rangle = -\langle Y,[\overline X,Z]\rangle, \qquad \forall X,Y,Z \in \mathfrak g_{\C}.$$

Now, given a basis $\{T_1,\dots, T_{\mathrm{rank}\mathfrak g}\}$ for $\mathfrak t$, the
endomorphisms $\mathrm{ad}_{T_1},\dots, \mathrm{ad}_{T_{\mathrm{rank}\mathfrak g}}$ commute and are
skew-Hermitian with respect to the ad-invariant inner product $\langle\cdot,\cdot\rangle$ on
$\mathfrak g_{\C}$. Hence, they are simultaneously diagonalizable, and we may find common
(one-dimensional) eigenspaces $\mathfrak g_{\alpha}=\operatorname{span}_\C \{ L_{\alpha} \}$ and a decomposition
\begin{equation}\label{spectraldec}
  \mathfrak g_\C = \mathfrak t_\C \oplus \bigoplus_{\alpha \in \Delta_+} \mathfrak g_\alpha\oplus\overline{\mathfrak
  g}_{\alpha},
\end{equation}
where $\mathfrak t_{\C}=\mathfrak t\otimes\C$, $\overline{\mathfrak g}_{\alpha}=\mathfrak
g_{-\alpha}$ and $\Delta_+$ is a choice of a positive roots system.
Given a root system $\Delta$, one can choose a set of positive roots $\Delta_+$, which is a set with
the following properties: \begin{enumerate}
  \item for any $\alpha\in \Delta$, exactly one between $\alpha$ and $-\alpha$ is contained in
  $\Delta_+$, \item for any two distinct $\alpha,\beta\in\Delta_+$ such that $\alpha+\beta\in
  \Delta$, we have $\alpha+\beta\in\Delta_+$.
\end{enumerate} Notice that $\overline{\alpha}=-\alpha$, so we denote
$\Delta_-=\overline{\Delta}_+$. We have
$\Delta = \Delta_+ \cup\Delta_-$ and $\Delta_+ \cap \Delta_- = \emptyset$, which also implies that
$\dim \mathfrak g_\C - \dim \mathfrak t_\C$ is an even number, a fact that we will use later.
An element in $\Delta_+$ is called a simple root if it cannot be written as the sum of two elements
in $\Delta_+$. It is common to think $\alpha\in\Delta_+$ as an element in $\mathfrak t_{\C}^*$,
namely $\alpha=(\alpha(T_1),\dots,\alpha(T_{\mathrm{rank}\mathfrak g}))$, where $\alpha(T_j)\in
i\mathbb R$ for $j=1,\dots,\mathrm{rank}\mathfrak g$. According to the notation adopted above, we
denote by $\mathfrak g_0$ the Cartan subalgebra $\mathfrak t_{\mathbb C}$, that is, the
$\mathrm{rank} \mathfrak g$-dimensional eigenspace relative to 0.
\begin{rmk}
  \label{rem:orthogonality_rel}
  We have $$\mathrm{ad}_{T_j}(L_{\alpha})=[T_j,L_{\alpha}]=\alpha(T_j)L_{\alpha},$$ so that
  $[T_j,L_{\alpha}]\in\mathfrak g_{\alpha}$. Moreover, we have
  $$\mathrm{ad}_{T_j}([L_{\alpha},L_{\beta}])=[T_j,[L_{\alpha},L_{\beta}]]=-[L_{\alpha},[L_{\beta},T_j]]-[L_{\beta},[T_j,L_{\alpha}]]=$$
  $$=-[L_{\alpha},-\beta(T_j)L_{\beta}]-[L_{\beta},\alpha(T_j)L_{\alpha}]=(\alpha+\beta)(T_j)[L_{\alpha},L_{\beta}],$$
  so that $[L_{\alpha},L_{\beta}]\in \mathfrak g_{\alpha+\beta}$ if $\alpha + \beta$ is a root,
  $[L_\alpha,L_\beta] \in \mathfrak g_0 = \mathfrak t_\C$ if $\beta = -\alpha$,
  and $[L_{\alpha}, L_{\beta}] = 0$ otherwise. From ad-invariance, we also easily verify that
  $\mathfrak g_\alpha \perp \mathfrak t_\C$ for all $\alpha \in \Delta.$ In fact, for $T, T' \in \mathfrak t$
  with $\alpha(T) \neq 0$ (such $T$ exist because $\alpha$ is a root and so $\alpha \neq 0$), we
  have
  \[
  \alpha(T)\langle L_\alpha, T' \rangle = \langle [T, L_\alpha], T' \rangle = - \langle L_\alpha,
  [T, T'] \rangle = 0.
  \]
\end{rmk}
Here, we are interested in the odd-dimensional case, so we will suppose $\mathrm{dim} \mathfrak g
= 2k+1$, which also implies that $\mathrm{rank}$ is odd and so we write $\mathfrak g = 2r+1$ for some $0 \leq r \leq k$. Now, suppose that
$\mathrm{dim} \mathfrak v = \mathrm{rank} \mathcal V = k+1$, that is, $\mathfrak v$ has minimal
dimension among the elliptic subalgebras of $\mathfrak g_\C$.
We are interested in giving a complete algebraic decomposition of these types of left-invariant elliptic structures of minimal rank.

Regarding left-invariant CR structures $\mathfrak s$ of hypersurface type (i.e. of maximal rank $=k$), according to the classification by \cite{charbonnel_classification_2004}, the only possible cases are of type CR0 or CR1. In the first case $\mathfrak s$ is conjugated (under the adjoint action of $G$ on $\mathfrak g_{\mathbb C}$ by $\operatorname{Ad}(g)$) to a subalgebra of the form $$\mathfrak m\oplus\bigoplus_{\alpha\in\Delta_+}\mathfrak g_{\alpha},$$ where $\mathfrak m\subset \mathfrak t_{\mathbb C}$ is a maximal CR subalgebra such that $\mathfrak m \cap \overline{\mathfrak m} =\{0\}$, $\Delta_+$ is a choice of positive roots and $\mathfrak g_{\alpha}$ is the eigenspace relative to $\alpha\in\Delta_+$. In the CR1 case, $\mathfrak s$ is conjugated to a subalgebra of $\mathfrak g_\C$ of the form $$\mathfrak m\oplus\bigoplus_{\beta\in\Delta_+\setminus\{\alpha\}}\mathfrak g_{\beta}\oplus\mathrm {span}_{\mathbb C}\{t+x\},$$ where $\mathfrak m\subset \mathrm{ker} \alpha$ is a maximal CR subalgebra of $\mathfrak t_{\mathbb C}$ such that $\mathfrak m \cap \overline{\mathfrak m} = \{0\}$, $\alpha\in\Delta_+$ is a simple root, $x \in \mathfrak g_{\alpha} \setminus \{0\}$, and $t \in \mathfrak t \setminus \{0\}$.

In the following, we introduce a notion adapted from \cite{jacobowitz_transversely_2000}.

\begin{dfn}
  A CR structure of maximal rank $\mathfrak s$ is said to be \emph{dominated by an elliptic
  structure of minimal rank} $\mathfrak v$ if $\mathfrak s \subset \mathfrak v$ and $\mathfrak v =
  \mathfrak s \oplus \operatorname{span}_\C \{ X \}$ for some $X \in \mathfrak g \setminus \{0\}$.
\end{dfn}

It is obvious from the definition that a CR structure of type CR0 is always dominated by an elliptic
structure of minimal rank. In fact, if $\mathfrak s = \mathfrak m \oplus \bigoplus_{\alpha \in \Delta_+}
\mathfrak g_\alpha$ we can take a $X \in \mathfrak t \setminus (\mathfrak m \oplus \overline{\mathfrak
m})$ and it is easy to check that
\[
  \mathfrak v = \mathfrak m \oplus \operatorname{span}_\C \{X\}
  \oplus \bigoplus_{\alpha \in \Delta_+} \mathfrak g_\alpha
\]
is elliptic and dominates $\mathfrak s$.

It is also possible to show that some CR structures of type CR1 are dominated by an elliptic structure of minimal rank. In fact, if
\[
 \mathfrak s = \mathfrak m \oplus \bigoplus_{\beta \in \Delta_+, \beta \neq \alpha} \mathfrak
  g_{\beta} \oplus \operatorname{span}_\C \{x+t\}
\]
with $\alpha \in \Delta_+$ a simple root, $\mathfrak m \subset \ker \alpha$ a subspace with
$\mathfrak m \cap \overline{\mathfrak m} = \{0\}$, a non-zero $x \in \mathfrak g_\alpha$ and $t \in \mathfrak t \setminus \{0\}$ and $\alpha(t) \neq 0$, then $\operatorname{span}_\C \{t\} \oplus \mathfrak m$ is elliptic in $\mathfrak t_\C$ and
\[
  \mathfrak v = \operatorname{span}_\C \{t\} \oplus \mathfrak m\oplus\bigoplus_{\beta \in \Delta_+,\beta\neq\alpha}\mathfrak g_{\beta} \oplus \operatorname{span}_\C \{x+t\} = \operatorname{span}_\C \{t\} \oplus \mathfrak m \oplus \bigoplus_{\beta \in \Delta_+}
\mathfrak g_\beta
\]
dominates $\mathfrak s$.

We prove that a CR structure of type CR1 cannot be dominated by a left-invariant elliptic structure
of minimal rank if the real direction is orthogonal (with respect to $\langle \cdot,\cdot\rangle$)
to the CR structure. First, we prove a lemma that expresses $[L_{\alpha},\overline
L_{\alpha}]$ in terms of the Cartan subalgebra. Such a decomposition is a key step in the proof.
Specifically, given a decomposition as in \eqref{spectraldec} with $\mathfrak t =
\operatorname{span}\{T_1, \ldots, T_{2r+1}\}$ consisting of pairwise orthonormal vectors, that is,
$\langle T_j, T_k \rangle = \delta_{jk}$, we show that $[L_{\alpha},\overline L_{\alpha}]$ is a
linear combination of $\{T_1,\dots,T_{2r+1}\}$ whose coefficient along $T_j$ is $\alpha(T_j)$.

\begin{lem}
  \label{lem:decomposition_wrt_CSA}
  Let $\{T_1, \ldots, T_{2r+1}\}$ be an orthonormal basis for $\mathfrak t$ with respect to the
  ad-invariant inner product, and suppose that each $L_\alpha$, for $\alpha \in \Delta_+$, is
  normalized so that $\langle L_{\alpha}, L_{\alpha}\rangle = \overline{ \langle L_{\alpha}, L_{\alpha}\rangle } = \langle \overline L_{\alpha}, \overline  L_{\alpha}\rangle = 1$. Then, for every $\alpha \in \Delta_+$,
  \[
  [L_{\alpha}, \overline L_{\alpha}] = \sum_{j=1}^{2r+1} \alpha(T_j) T_j.
  \]
  In particular, the component of $[L_{\alpha},\overline L_{\alpha}]$ in the $T_j$-direction equals
  $\alpha(T_j)$.
\end{lem}

\begin{proof}
  By Remark~\ref{rem:orthogonality_rel}, we have $[L_{\alpha},\overline L_{\alpha}] \in \mathfrak g_{0}
  = \mathfrak t_{\mathbb C}$ is a linear combination of the $T_j$'s. We write $[L_{\alpha},\overline
  L_{\alpha}]=\sum_{j=1}^{2r+1}e_j^{\alpha}T_j$. By ad-invariance, we have $$e_j^{\alpha} = \langle [L_{\alpha},\overline
  L_{\alpha}],T_j\rangle = -\langle \overline L_{\alpha},[\overline L_{\alpha}, T_j]\rangle
  = -\langle \overline L_{\alpha},-[T_j, \overline L_{\alpha}]\rangle
  =$$
  $$= \langle \overline L_{\alpha},\overline {\alpha(T_j)}\ \overline L_{\alpha}\rangle
  = \alpha(T_j) \underbrace{\langle \overline L_{\alpha}, \overline L_{\alpha} \rangle}_{=1}
  = \alpha(T_j),$$ where, in the last line, we used the fact that the inner product is linear in the
  first variable, and anti-linear in the second one.
\end{proof}
\begin{rmk}
  \label{rmk:on_properties_L_alpha}
  Writing $L_\alpha = X_\alpha - iY_\alpha$, with $X_\alpha,Y_\alpha \in \mathfrak g$, and using
  $\mathfrak g_\alpha \perp \mathfrak g_{-\alpha}$, we have that $\langle
  L_{\alpha},\overline L_{\alpha} \rangle =0$ and that $\langle
  L_{\alpha},L_{\alpha}\rangle =1$ is equivalent to
  $$ \begin{cases}
    \langle X_{\alpha}, X_{\alpha}\rangle = \langle Y_{\alpha}, Y_{\alpha}\rangle = \frac{1}{2},\\ \langle X_{\alpha}, Y_{\alpha}\rangle = \langle Y_{\alpha}, X_{\alpha} \rangle = 0.
  \end{cases}$$

  From now on, we assume that each $L_\alpha$ is normalized, that is, we assume
  \[
  \langle L_{\alpha}, L_{\alpha}\rangle   = 1, \qquad \forall \alpha \in \Delta_+. \]
\end{rmk}

In the following, we prove that CR structures of type CR1 cannot be dominated by a left-invariant
elliptic structure of minimal rank if the real direction is orthogonal to the CR1 structure.

\begin{lem}
  \label{lem:cr1_is_never_dominated}
  A CR Lie algebra of type CR1 on an odd-dimensional, connected and compact Lie group cannot be dominated by a
  left-invariant elliptic structure of minimal rank if the real direction is orthogonal to the CR
  structure.
\end{lem}
\begin{proof}
  Let $\mathfrak s$ be a CR structure of type CR1. By the Charbonnel-Khalgui classification,
  $\mathfrak s$ is conjugated to
  \[
  \Theta(\mathfrak m,\alpha,x,t) = \mathfrak m\oplus\bigoplus_{\beta\in\Delta_+,\beta\neq\alpha}\mathfrak
  g_{\beta} \oplus \operatorname{span}_\C \{x+t\}
  \]
  where $\alpha \in \Delta_+$ is a simple root, $\mathfrak m \subset \ker \alpha$ is a subspace with
  $\mathfrak m \cap \overline{\mathfrak m} = \{0\}$, a non-zero $x \in \mathfrak g_\alpha$ and non-zero $t \in \mathfrak t$.
  Since conjugation preserves orthogonality, the CR1 type (by definition) and the domination
  property, it is enough to show that
  $
  \Theta(\mathfrak m,\alpha,x,t)
  $
  cannot be dominated. Without loss of generality, we assume that $\langle x, x \rangle = 1$.

  By contradiction, we assume that there is a left-invariant elliptic structure $\mathfrak v$
  dominating $\mathfrak s$, with $\mathfrak v= \mathfrak s \oplus \operatorname{span}_\C \{ X \}$,
  where $\langle X,Z\rangle=0$ for all $Z\in\mathfrak s$ and $0\neq X\in\mathfrak
  v\cap\overline{\mathfrak v}$.
  Notice that if $X \perp \mathfrak s$, then $X \perp \overline{\mathfrak s}$. In fact, if for all
  $Z \in \mathfrak s$, we have $$0 = \langle X, Z \rangle = \langle X, \mathrm{Re} Z + i\mathrm{Im}
  Z \rangle = \langle X, \mathrm{Re} Z \rangle -i\langle X, \mathrm{Im} Z \rangle$$ and thus $
  \langle X, \mathrm{Re} Z \rangle = 0$ and $\langle X, \mathrm{Im} Z \rangle = 0$, which means that
  $X \perp \overline Z$. We get $X \perp \mathfrak m \oplus \overline{\mathfrak m} = \operatorname{ker} \alpha$, that is $X
  \perp \ker \alpha$.
  Thus, since $\ker\alpha \oplus \bigoplus_{\beta \in \Delta_+ \setminus \{\alpha\}} \mathfrak g_{\beta}\oplus\mathfrak
  g_{-\beta}\subset\mathfrak s\oplus\overline{\mathfrak s}$, then we may write
  \[
  X = X_0 + a \mathrm{Re}(x) + b \mathrm{Im}(x) \in (\ker\alpha)^{\perp}\oplus\mathfrak g_{\alpha}\oplus\mathfrak g_{-\alpha}, \qquad a, b\in\mathbb R,
  \]
  where $X_0\in(\ker\alpha)^{\perp}\subset\mathfrak t_{\mathbb C}$.
  We have $\langle X,x+t\rangle=0$.
  Since $t$ is real, we may write $t=\kappa X_0+T,$ for some $\kappa \in \mathbb R$, $X_0\in
  (\ker\alpha)^{\perp},$ and a real $T\in\ker\alpha$. Then $$0=\langle X,x+t\rangle=\langle
  X,\mathrm{Re} x+i\mathrm{Im} x + \kappa X_0+T\rangle$$ implies (we already know $\langle X,T\rangle=0$)
  that $$\begin{cases}
    \langle X,\mathrm{Re} x\rangle+\kappa \langle X,X_0\rangle=0\\
    \langle X,\mathrm{Im} x\rangle=0.
  \end{cases}$$

  Notice that, from the previous orthogonality relations, $X_0 = 0$ would imply $a = b = 0$ and $X =
  0$, a contradiction. So $X_0 \neq 0$, and, dividing $X$ by $\langle X_0, X_0 \rangle^{1/2}$ if
  necessary, we can assume $\langle X_0, X_0 \rangle = 1$.

  We have $\langle X_0, \operatorname{Re} x \rangle = \langle X_0, \operatorname{Im} x \rangle = 0$
  and from Remark \ref{rmk:on_properties_L_alpha} we get $\langle X, \operatorname{Re} x \rangle =
  \langle a \operatorname{Re} x, \operatorname{Re} x \rangle = a/2$. Thus, we deduce that $a = - 2\kappa 
  \langle X_0, X_0 \rangle = -2\kappa $ and that $b = 0.$

  Since $X_0 \neq 0$ and $X_0 \in (\operatorname{ker} \alpha)^\perp$, we write
  $\frac{1}{\alpha(X_0)}[X,x+t]\in\mathfrak v$ as a linear combination (with coefficients
  $c,d,e\in\mathbb C$) of $X$ and $x+t$ and some $M\in\mathfrak m$. In the computations, we use
  $[X_0,T]= 0$ because $X_0$ and $T$ belong to $\mathfrak t$ and $[x,T] =- \alpha(T)x = 0$ because
  $T \in \ker \alpha$ and similarly $[\overline x,T]=0$. By Lemma \ref{lem:decomposition_wrt_CSA} we
  have $[\overline x , x ] = -\alpha(X_0)X_0$.

  $$ \begin{aligned} \frac{1}{\alpha(X_0)}[X,x+t] & =\frac{1}{\alpha(X_0)}[X_0-\kappa (x+\overline x),x+\kappa X_0+T] \\
    & =\frac{1}{\alpha(X_0)}\biggl(\underbrace{[X_0,x]}_{\alpha(X_0)x}-\kappa \underbrace{[\overline x,x]}_{-\alpha(X_0)X_0}-\kappa^2\underbrace{[x,X_0]}_{-\alpha(X_0)x}-\kappa^2\underbrace{[\overline x,X_0]}_{\alpha(X_0)\overline x}\biggr) \\
    & = \kappa X_0+(1+\kappa ^2)x -\kappa^2\overline x \in \mathfrak t_\C \oplus \mathfrak g_\alpha \oplus \mathfrak g_{-\alpha}\\
    & =c(\underbrace{X_0-\kappa (x+\overline x)}_{X})+d(\underbrace{x+\kappa X_0+T}_{x+t})+eM \\
    & =(c+d\kappa )X_0+(-c\kappa +d)x-c\kappa \overline x+dT+eM.
  \end{aligned} $$

  Comparing term by term ($X_0,x,\overline x$ are linearly independent), and using that $\kappa $ is real,
  we get some identities for the coefficients and deduce that $$c=\kappa ,\quad d=1+2 \kappa ^2, \quad
  \kappa (1+2\kappa ^2)=0,$$ and we get $\kappa = c = 0, d=1$, so that $T$ must be zero (notice that if $T \neq 0$ then
  $T+eM=0$ has no solutions $(e,M)\in\mathbb C\times\mathfrak m$ since $T$ is real and $\mathfrak
  m\cap\mathfrak t=0$). But then, $t= \kappa X_0+T=0$, which is a contradiction, and the proof is complete.
\end{proof}

\begin{rmk}
  Notice that, in general, it is possible for a CR1 structure to be dominated by an elliptic one.
  Consider for instance the following CR1 structure on $\mathrm{SU}(2)$ $$\mathfrak s=
  \operatorname{span}_\C \{ Z = X - iY + \tau T \},\qquad \tau \in \mathbb R \setminus \{0\}.$$ Then
  $\mathfrak s$ is dominated by the elliptic structure $\mathfrak v = \operatorname{span}_\C \{ T,Z
  \} = \operatorname{span}_\C \{ T, X-iY \}$. Notice that $\langle T, Z \rangle \neq 0$, so this is
  not in contradiction with Lemma \ref{lem:cr1_is_never_dominated}.
\end{rmk}

\begin{dfn}
  An elliptic Lie algebra $\mathfrak v$ is said to be of \emph{T-type} if it is conjugate (up to the
  adjoint action of $G$ on $\mathfrak g_{\C}$) to
  \[
  \mathfrak e \oplus \bigoplus_{\alpha \in \Delta_+} \mathfrak g_\alpha
  \]
  in which $\mathfrak e \subset \mathfrak t_\C.$ The subalgebra $\mathfrak e$ is called the toral
  part of $\mathfrak v$.
\end{dfn}

Theorem \ref{thm:decomposition} shows that every left-invariant elliptic structure of minimal rank
on an odd-dimensional compact Lie group is of T-type.
\begin{rmk}
    It follows from Theorem \ref{thm:decomposition} that any left-invariant elliptic structure of minimal rank $\mathfrak v=\mathfrak e\oplus\bigoplus_{\alpha\in\Delta_+}\mathfrak g_{\alpha}$ dominates a CR0 structure. To see this, it suffices to notice that, setting $\mathfrak e=\mathrm{span}_{\C}\{X\}\oplus\mathfrak s$ with $0\neq X\in\mathfrak v\cap\overline{\mathfrak v}$, we have that $\mathfrak s\oplus\bigoplus_{\alpha\in\Delta_+}\mathfrak g_{\alpha}$ is still a Lie algebra, and it is forced to be of CR0 type.
\end{rmk}

\section{Decomposition of elliptic structures of minimal rank}
\label{sec:algebraic_decomposition}

In this section, we prove Theorem \ref{thm:decomposition}. Along the proof, we will state and prove
an auxiliary lemma (Lemma \ref{lem:dimensional_calculation}) and a proposition (Proposition
\ref{prop:sisasubalgebra}).

\begin{proof}[Proof of Theorem \ref{thm:decomposition}]
  Since $\mathfrak v\subset\mathfrak g_{\C}$ is an elliptic subalgebra of dimension $k+1$, the
  intersection $\mathfrak v \cap \overline{\mathfrak v}$ has dimension one. Let $X \in \mathfrak g
  \setminus \{0\}$ be any vector spanning $\mathfrak v \cap \overline{\mathfrak v}$. Since
  $\operatorname{span}_\R \{X\}$ is abelian, there exists a maximal abelian subalgebra $\mathfrak t
  \subset \mathfrak g$ containing $X$. Let $X, T_1,\dots, T_{2r}$ be linearly independent
  vector fields spanning $\mathfrak t$. By using an orthonormalization process, we may assume
  that $\{X, T_1,\dots, T_{2r}\}$ is an orthonormal basis for $\mathfrak t$.

  The complexification of $\mathfrak t$, which we denote by $\mathfrak t_\C$, is a Cartan subalgebra
  of $\mathfrak g_\C$ \cite[Lemma 12.2.1]{hilgert_structure_2012}. The root space decomposition
  gives us
  \[
  \mathfrak g_{\mathbb C}=\mathfrak t_{\mathbb C}\oplus\bigoplus_{\alpha\in\Delta}\mathfrak
  g_{\alpha},
  \]
  in which $\Delta$ is the set of roots with respect to $\mathfrak t_\C$.

  Now, we consider the adjoint map $\mathrm{ad}_X:\mathfrak g_{\mathbb C} \rightarrow \mathfrak{g}_{\mathbb
  C}$ and notice that $\mathrm{ker}(\mathrm{ad}_X)$ is a subalgebra of $\mathfrak g_{\mathbb
  C}$: if $A, B \in \mathrm{ker}(\mathrm{ad}_X)$ by Jacobi identity we have
  \[
  [X,[A,B]]=-[A,[B,X]]-[B,[X,A]]=0,
  \]
  so that $[A, B] \in \mathrm{ker} (\mathrm{ad}_X)$.

  Let $Z = T + \sum_{\alpha \in \Delta} L_\alpha$, with $T \in \mathfrak{t}_\C$ and $L_\alpha \in \mathfrak g_\alpha$,
  and notice that
  \[
  \mathrm{ad}_X(Z) = [X,T] + \sum_{\alpha \in \Delta} [X,L_\alpha] = \sum_{\alpha \in \Delta}
  \alpha(X)L_{\alpha}
  \]
  and so we have the following splitting
  \[
  \mathfrak g_{\C} = \ker({\mathrm{ad}_X}) \oplus \bigoplus_{\alpha \in \Delta,~ \alpha(X) \neq 0}
  \mathfrak g_{\alpha}.
  \]
  We define
  \[
  \mathfrak f \doteq \bigoplus_{\alpha \in \Delta,~\alpha(X) \neq 0}\mathfrak g_{\alpha},
  \]
  and
  \[
  l \doteq |\{\alpha\in\Delta_+:\alpha(X)\neq 0\}|\leq k-r
  \]
  for some choice of positive roots $\Delta_+$. Notice that $\dim \mathfrak f = 2l$.

  \begin{lem}
    \label{lem:dimensional_calculation}
    We have
    $$\dim(\mathfrak v\cap\ker(\mathrm{ad}_X))=k+1-l.$$
  \end{lem}

  \begin{proof}
    Since $\mathfrak v$ is elliptic and $X \notin \mathfrak f$, we have
    \[ \begin{aligned}
      2\dim(\mathfrak v\cap \mathfrak f) & = \dim(\mathfrak v\cap \mathfrak f) +\dim(\overline{\mathfrak v}\cap \mathfrak f) \\
      & \leq \dim(\mathfrak f) +\underbrace{\dim(\mathfrak v\cap\overline{\mathfrak v}\cap\mathfrak f)}_{=0} \\
      & \leq 2l.
    \end{aligned} \]
    Since $\mathrm{Im}(\mathrm{ad}_X|_{\mathfrak v})\subset\mathfrak v\cap\mathfrak f$, we have
    \[ \begin{aligned}
      \dim(\mathfrak v\cap \mathrm{ker}(\mathrm{ad}_X)) &= \dim(\ker{\mathrm{ad}_X|_{\mathfrak v}}) \\
      &= \dim(\mathfrak v)-\dim(\mathrm{Im}(\mathrm{ad}_X|_{\mathfrak v})) \\
      & \geq \dim{\mathfrak v}-\dim(\mathfrak v\cap\mathfrak f) \\
      & \geq k+1-l.
    \end{aligned} \]
    On the other hand, since $\mathfrak v+\overline{\mathfrak v}=\mathfrak g_{\C}$ and
    $X\in\ker{\mathrm{ad}_X}$, we have
    \[ \begin{aligned}
      2\dim(\mathfrak v\cap \mathrm{ker}(\mathrm{ad}_X)) & \leq \dim(\mathrm{ker}(\mathrm{ad}_X))+\underbrace{\dim(\mathfrak v\cap\overline{\mathfrak v}\cap\mathrm{ker}(\mathrm{ad}_X))}_{=1} \\
      & =2(k-l) + 1 + 1 \\
      & = 2(k - l + 1).
    \end{aligned} \]
    By combining the two previous inequalities, we obtain the desired equality $\dim(\mathfrak v\cap
    \mathrm{ker}(\mathrm{ad}_X))=k-l+1$.
  \end{proof}

  By Lemma \ref{lem:dimensional_calculation}, we can write
  \[
  \mathfrak v = \mathrm{span}_{\C} \{X, W_1,\dots W_{k-l}, W_{k-l+1},\dots, W_k \},
  \]
  with all $W_j$ orthogonal to $X$, and $\mathfrak v \cap \mathrm{ker}(\mathrm{ad}_X) =
  \mathrm{span}_{\C}\{ X, W_1,\dots, W_{k-l}\}$ is a subalgebra of $\mathfrak v$ and an elliptic
  subalgebra of $\mathrm{ker}(\mathrm{ad}_X)$.

  \begin{prp}
    \label{prop:sisasubalgebra}
    The set $\mathfrak s \doteq \mathrm{span}_{\C}\{W_1,\dots,W_{k-l}\}$ is a subalgebra of
    $\mathfrak g_{\C}$.
  \end{prp}

  \begin{proof}

    We already know that $\mathfrak v\cap\ker(\mathrm{ad}_X)=\mathrm{span}_{\C}\{ X,W_1,\dots,
    W_{k-l}\}$ is a subalgebra of $\mathfrak g_{\C}$, being the intersection of two Lie
    subalgebras. Hence, we just have to verify that for each $i,j=1,\dots,k-l$, the component of
    $[W_i,W_j]$ in the $X$'s direction is zero. For $i=1,\dots, k-l$, we can write $W_i$ as a linear
    combination $$W_i=\sum_{u=1}^{2r}a_u^iT_u+\sum_{\alpha\in\Delta_+:\alpha(X)=
    0}b^i_{\alpha}L_{\alpha}+c^i_{\alpha}\overline L_{\alpha},$$ for some coefficients
    $a_u^i,b_{\alpha}^i,c_{\alpha}^i\in\mathbb C$. Let us show this fact: in principle,
    we can write $W_i$ as
    $$W_i=\sum_{u=1}^{2r}a_u^iT_u+\sum_{\alpha\in\Delta_+}b_{\alpha}^iL_{\alpha}+c_{\alpha}^i\overline
    L_{\alpha},$$ that is, as a linear combination of the basis elements (notice that by construction
    each $W_i$ is orthogonal to $X$). But since $W_i\in\ker(\mathrm{ad}_X)$ we have
    $$0=[X,W_i]=\sum_{u=1}^{2r}a_u^i\underbrace{[X,T_u]}_{=0}+\sum_{\alpha\in\Delta_+}b_{\alpha}^i\underbrace{[X,L_{\alpha}]}_{=\alpha(X)L_{\alpha}}+c_{\alpha}^i\underbrace{[X,\overline
    L_{\alpha}]}_{=-\alpha(X)\overline L_{\alpha}},$$ so that we must have
    $b_{\alpha}^i=c_{\alpha}^i=0$ whenever $\alpha(X)\neq0$, because the set
    $\{L_{\alpha},\overline{L}_{\alpha}, \alpha\in\Delta_+,\alpha(X)\neq 0\}$ is made of linearly
    independent vectors.

    Now we want to study the bracket $[W_i,W_j]$ for $i,j=1,\dots, k-l$. By Remark
    \ref{rem:orthogonality_rel}, for any $T\in\mathfrak t$ and $\alpha,\beta\in\Delta$ such that
    $\alpha\neq -\beta$, the Lie brackets $[T,L_{\beta}]$ and $[L_{\alpha},L_{\beta}]$ belong to
    $\bigoplus_{\gamma \in \Delta} \mathfrak g_\gamma,$
    so they have no components in the $X$'s direction. The only contributions in the toric
    directions can come from the brackets of type $[L_{\alpha},\overline L_{\alpha}]$. By Lemma
    \ref{lem:decomposition_wrt_CSA}, the coefficient of $[L_{\alpha},\overline L_{\alpha}]$ in the
    $X$'s direction is proportional to $\alpha(X)$, for any $\alpha\in\Delta_+$. Since the $W_j$'s
    are written as above, for each $i,j \in \{1, \dots, k-l\}$ the only brackets of type
    $[L_{\alpha},\overline L_{\alpha}]$ appearing in $[W_i,W_j]$ are those for which $\alpha(X)=0$.
    This concludes the proof.
  \end{proof}

  \begin{cor}
    \label{cor:s_is_cr0_asubalgebra}
    The algebra $\mathfrak s$ is a maximal-rank CR subalgebra of $\mathrm{ker}(\mathrm{ad}_X)$ and
    it is dominated by the elliptic subalgebra $\mathfrak v\cap \mathrm{ker}(\mathrm{ad}_X)$.
  \end{cor}

  \begin{proof}
    Since $\mathfrak s \subset \mathfrak v$, we have $\mathfrak s \cap \overline {\mathfrak s}
    \subset \mathfrak v \cap \overline {\mathfrak v} = \operatorname{span}_\C \{X\}$ but since $\{X,
    W_1, \ldots, W_{k-l}\}$ is linearly independent, we conclude $\mathfrak s \cap \overline {\mathfrak
    s} = \{0\}$ and $\mathfrak s$ is a CR subalgebra. It is of hypersurface type because $\dim_\C
    \operatorname{ker}(\operatorname{ad}_X) = 2k + 1 -2l$ and $\dim_\C \mathfrak s = k-l$ and so
    $2\dim_\C \mathfrak s + 1 = \dim_\C \operatorname{ker}(\operatorname{ad}_X)$. It is clearly
    dominated by $\mathfrak v \cap \operatorname{ker} (\operatorname{ad}_X)$ since 
    $\mathfrak v \cap \operatorname{ker}
    (\operatorname{ad}_X) = \mathfrak s \oplus \operatorname{span}_\C \{X\}$ and $X \neq 0$. And
    $\mathfrak s \oplus \operatorname{span}_\C \{X\}$ is elliptic of minimal rank in
    $\operatorname{ker}(\operatorname{ad}_X)$.
  \end{proof}

  The strategy now is to apply the Charbonnel-Khalgui classification theorem to $\mathfrak s$ seen
  as a CR structure defined on a closed subgroup of $G$.

  First we consider $K \doteq \exp(\mathrm{ker}(\mathrm{ad}_X|_{\mathfrak g}))$. Since $K$ contains
  the maximal torus $\mathbb T = \exp(\mathfrak t)$, by 
  \cite[Corollary 14.5.6]{hilgert_structure_2012}, we have that $K$ is closed in $G$ and thus compact. 
  Second, by the
  Charbonnel-Khalgui classification theorem, we have $\mathfrak s$ is either of type CR0 or of type
  CR1. That is, $\mathfrak s \subset\mathrm{ker}(\mathrm{ad}_X)$ is a maximal-rank CR Lie
  subalgebra dominated by the elliptic subalgebra $\mathfrak v\cap\mathrm{ker}(\mathrm{ad}_X)$.
  Also, notice that, by construction, $X$ is orthogonal to $\mathfrak s$. Then, by Lemma
  \ref{lem:cr1_is_never_dominated}, $\mathfrak s$ must be of CR0 type. Then there is $g \in K$ such that
  \[
    \mathfrak s = \mathrm{Ad}_g(\mathfrak m \oplus \bigoplus_{\lambda\in D_+} \mathfrak g_{\lambda})
  \]
  where $\mathfrak m\subset\mathfrak
  t_{\mathbb C}$ is a CR subalgebra of maximal rank. Since the adjoint action of $G$ on $\mathfrak
  g_{\C}$ permutes the Cartan subalgebras, and $\operatorname{Ad}_g X = X$, since $X$ is central in $\operatorname{ker} \operatorname{ad}_X$, we have that 
  $\mathfrak t_\C' \doteq \operatorname{Ad}_g (\mathfrak t_\C)$ is a Cartan subalgebra of 
  $\mathfrak g_\C$ containing $X$, and $\mathfrak v \cap \mathfrak t_\C' = \operatorname{span}_\C \{ X \} \oplus \operatorname{Ad}_g (\mathfrak m)$ is an elliptic subalgebra of $\mathfrak t_\C'.$
  Let $\mathfrak e \doteq \mathfrak v \cap \mathfrak t_\C' = \operatorname{span}_\C \{X\} \oplus \operatorname{Ad}_g(\mathfrak m)$ 
  and let $\Delta'$ be the set of roots of $\mathfrak g_\C$ with respect to $\mathfrak t_\C'$. 
  For each $\beta \in \Delta'$, we denote by $\mathfrak g_\beta$ the corresponding root space 
  and let $S = \{\beta \in \Delta': \mathfrak g_\beta \subset \mathfrak v\}$. In the following, we show that
  \begin{equation}
  \label{eq:root_space_decomp_v_ell}
      \mathfrak v = \mathfrak e \oplus \bigoplus_{\beta \in S} \mathfrak g_\beta
  \end{equation}
  and that $S$ is a system of positive roots of $\Delta'.$

  Notice that no root of $\mathfrak t_\C'$ vanishes on $\mathfrak e$ and that distinct roots of $\mathfrak t_\C'$ have distinct restrictions to $\mathfrak e$. This follows directly from the fact that the restriction of a root to $\mathfrak t'$ is purely imaginary. If $\gamma$ is a root or the difference of two distinct roots and $\gamma$ vanishes on $\mathfrak e$, then it would also vanish on $\overline{\mathfrak e}$ and thus on $\mathfrak e + \overline{\mathfrak e} = \mathfrak t_\C'$, so $\gamma = 0$, which is impossible in any case.

  As a consequence, we have that the weight spaces of $\mathfrak g_\C$ under the action of $\operatorname{ad}(\mathfrak e)$ are $\mathfrak t_\C'$ for 0 and the root spaces $\mathfrak g_\beta$. Since $\mathfrak v$ is $\operatorname{ad}(\mathfrak e)$-invariant and $\mathfrak e$ is abelian (and thus nilpotent), we can apply \cite[Lemma 6.1.3]{hilgert_structure_2012} to decompose $\mathfrak v$ into the sum of its $\operatorname{ad}(\mathfrak e)$-weight spaces, and each of them is contained in the corresponding weight space of $\mathfrak g_\C$. Since the root spaces are one-dimensional, we have that $\mathfrak v \cap \mathfrak g_\beta$ is either $\{0\}$ or $\mathfrak g_\beta$.
  We proved \eqref{eq:root_space_decomp_v_ell}.
  We now prove that $S$ is a system of positive roots. It is clear that $\beta$ and $-\beta$ cannot be both in $S$ otherwise $\mathfrak g_\beta \oplus \mathfrak g_{-\beta} = \mathfrak g_\beta \oplus \overline{\mathfrak g_{\beta}} \subset \mathfrak v \cap \overline{\mathfrak v} = \operatorname{span}_\C \{X\} \subset \mathfrak t'_\C$ which is a contradiction.

  And so we know that $|S| \leq |\Delta'|/2 = (\dim \mathfrak g_\C - \dim \mathfrak t_\C')/2 = k - r$.
  We also know from \eqref{eq:root_space_decomp_v_ell} that $k+1 = \dim \mathfrak v = \dim \mathfrak e + |S| = (r+1) + |S|$, since $\dim \mathfrak m = r$ and we conclude that $|S| = k - r$ and so $S$ contains exactly one of the $\beta$ and $-\beta$ for all $\beta \in \Delta'$. We also have that if $\beta, \gamma \in S$ and $\beta + \gamma \in \Delta'$, then $\mathfrak g_{\beta + \gamma} = [\mathfrak g_\beta, \mathfrak g_\gamma] \subset \mathfrak v$ by \cite[Lemma 6.3.5]{hilgert_structure_2012} and because $\mathfrak v $ is a subalgebra, and so $\beta + \gamma \in S$. And we proved that $S$ is a system of positive roots. We define $\Delta_+' \doteq S$ and $\mathfrak b_{\mathfrak t} = \bigoplus_{\beta \in \Delta_+'} \mathfrak g_\beta$ and \eqref{eq:root_space_decomp_v_ell} has the form $\mathfrak v = \mathfrak e \oplus \mathfrak b_{\mathfrak t}$ concluding the proof.
\end{proof} 

The following is an immediate consequence of the previous theorem.

\begin{cor}
  Every left-invariant elliptic Lie algebra of minimal rank on a compact Lie group of odd dimension
  is a solvable Lie algebra.
\end{cor}

\begin{proof}
  Let $\mathfrak v$ be a left-invariant elliptic Lie algebra of minimal rank on a compact Lie group
  $G$ of odd dimension. By using Theorem \ref{thm:decomposition}, we write $\mathfrak v = \mathfrak e
  \oplus \mathfrak b_{\mathfrak t}$. Now the proof follows from Remark \ref{rmk:on_properties_L_alpha}. In fact, $\mathfrak e$ is
  abelian, and therefore solvable, and $\mathfrak b_{\mathfrak t}$ is a solvable ideal of $\mathfrak v$ with $\mathfrak e \cong \mathfrak v / \mathfrak b_{\mathfrak t}$ solvable. Hence, $\mathfrak v$ is solvable
  \cite[Chapter 1, Proposition 2]{serre_complex_2001}.
\end{proof}

\section{Local structure of $\mathrm T$-type elliptic structures}
\label{sec:local_structure}

In this section, we show that $\mathrm T$-type elliptic Lie algebras on $G$ are locally trivial, that is, they
can be written as a product of a complex structure on an open subset of $\Omega = G / \T$ and an
elliptic one on the maximal torus $\T$. The main tools are
\cite[Lemma 4.2]{jahnke_closed_2026}, restated below as Lemma \ref{lem:LocalSection}, and the standard principal bundle structure of
$\pi : G \to \Omega$ recalled in Proposition \ref{prp:LocalIso}.

The following lemma gives us local sections for submersions compatible with an elliptic and a
complex structure. We will use it later for the projection $G \to \Omega = G / \T$.

\begin{lem}\label{lem:LocalSection}
  Let $M$ and $N$ be smooth connected manifolds. Assume that $M$ has an elliptic structure $\mathcal
  V$ and $N$ has a complex structure $\mathcal W$. Let $f:M\rightarrow N$ be a smooth surjective map such
  that $f_{\ast}(\mathcal V_x)=\mathcal W_{f(x)}$ for all $x\in M$. Then, for every $y_0\in N$ there
  exists an open neighborhood $V$ of $y_0$ and a map $\sigma:V\rightarrow M$ such that
  $f\circ\sigma(y)=y$ for all $y\in V$, and $\sigma_{\ast}(\mathcal W_y)\subset\mathcal
  V_{\sigma(y)}$ for all $y\in V$.
\end{lem}

\begin{proof}
  See \cite[Lemma 4.2]{jahnke_closed_2026}.
\end{proof}

The following is a standard result in the structure of Lie groups and is included here for the convenience of the reader.

\begin{prp}\label{prp:LocalIso}
  Let $G$ be a connected and compact Lie group and let $\mathbb T$ be a maximal torus in $G$. Then $G$
  is a principal fiber bundle with structure group $\mathbb T$ and base space $\Omega=G/\mathbb T$.
  In other words there exists a smooth right action $R:G\times\mathbb T\rightarrow G$ (which is the
  right action of $\mathbb T$ on $G$) such that \begin{enumerate}
    \item $\Omega=G/\mathbb T$ has a manifold structure making the projection $\pi:G\rightarrow
    \Omega$ smooth,
    \item $G$ is locally trivial, i.e. there is an open cover $\{V_j\}$ of $\Omega$ and
    diffeomorphisms $\Phi_j:\pi^{-1}(V_j)\rightarrow V_j\times\mathbb T$ such that
    $$\Phi^{-1}_j(y,st)=\Phi_j^{-1}(y,s)t,$$ for all $y\in V_j$ and all $s,t\in\mathbb T$.
  \end{enumerate}
\end{prp}
\begin{proof}
  This is a well-known result. See \cite[Theorem 10.1.10, Corollary
  10.1.11]{hilgert_structure_2012}. It is possible to find a covering $\{V_j\}$ of $\Omega$ and
  smooth sections $\sigma_j$ of the quotient map $\pi: G \rightarrow \Omega$ so that we obtain the
  following diffeomorphism onto an open subset of $G$ which we denote by $U_j$: $$V_j \times \mathbb T\ni(v,t)\mapsto\Psi_j(v,t)=\sigma_j(v)t\in\sigma_j(V_j)\mathbb
  T= U_j.$$

  The local trivialization follows by defining $\Phi_j=\Psi_j^{-1}$ and noting that
  $$\Phi_j^{-1}(y,st)=\Psi_j(y,st)=\sigma_j(y)st=\Psi_j(y,s)t=\Phi_j^{-1}(y,s)t.$$
\end{proof}

Notice that the last proposition holds if, instead of the maximal torus $\mathbb T$, we
consider any closed subgroup $K \subset G$.

\begin{prp}
  \label{prop:structurebundle}
  Let $G$ be a connected and compact Lie group endowed with a left-invariant elliptic structure $$\mathfrak
  v=\mathfrak e \oplus\mathfrak b_{\mathfrak t},$$ where $\mathfrak e \subset \mathfrak t_{\mathbb C}$
  is an elliptic subalgebra of the Cartan subalgebra $\mathfrak t_\C$ and $\mathbb T=\exp(\mathfrak
  t)\subset G$ is a maximal torus. Consider $\Omega=G/\mathbb T$ with the complex structure
  $\pi_{\ast}(\mathfrak v)$ (see \cite[Lemma 4.1.7]{jahnke_top-degree_2019} and \cite[Problem 2.57]{gadea_analysis_2013}), where $\pi:G\rightarrow \Omega$ is the natural projection and let
  $\Omega\times\mathbb T$ be endowed with the elliptic structure given by $\pi_{\ast}(\mathfrak
  v) \oplus \mathfrak e$. Then there exists a finite covering $\{V_j\}$ of $\Omega$ and local section
  $\sigma_j:V_j\rightarrow G$ such that
  \[
    \Psi_j:V_j\times\mathbb T\ni(v,t)\mapsto \sigma_j(v)t\in\sigma_j(V_j)\mathbb T=\pi^{-1}(V_j)
  \]
  satisfies $(\Psi_{j\ast})_{(v,t)}(\pi_{\ast}(\mathfrak v|_{\sigma_j(v)}) \oplus \mathfrak e|_t) =
  \mathfrak v_{\sigma_j(v)t}$.
\end{prp}

\begin{proof}
  By Proposition \ref{prp:LocalIso}, we have an open (finite, due to compactness) cover
  $\{V_j\}$ of $\Omega$ such that, for every $j$, the map $$\Psi_j:V_j\times\mathbb
  T\ni(v,t)\mapsto\sigma_j(v)t\in\sigma_j(V_j)\mathbb T=\pi^{-1}(V_j)$$ is a diffeomorphism onto
  $U_j=\pi^{-1}(V_j)$. Furthermore, from Lemma \ref{lem:LocalSection}, we can choose each $\sigma_j$ compatible with the structures. This will make each map $\Psi_j$ compatible with the involutive structures defined in their domains and ranges. In fact, let $X\oplus Y\in \pi_{\ast}(\mathfrak v|_{\sigma_j(v)}) \oplus\mathfrak e|_t$. We have
  \[ \begin{aligned}
    (\Psi_{j\ast})_{(v,t)}( X\oplus Y) & =(\Psi_{j\ast})_{(v,t)}(0\oplus Y)+(\Psi_{j\ast})_{(v,t)}(X\oplus 0) \\
  &=(L_{\sigma_j(v)})_{\ast}(Y)+(R_t)_{\ast}(\sigma_{j\ast}(X)).
  \end{aligned} \]
  The term $(L_{\sigma_j(v)})_{\ast}(Y)$ belongs to $\mathfrak v|_{\sigma_j(v)t}$ because of
  left-invariance. On the other hand, the term $\sigma_{j\ast}(X)$ belongs to $\mathfrak
  v_{\sigma_j(v)}$. We write $\sigma_{j\ast}(X)=X'+\sum_{\alpha\in\Delta_+}X_{\alpha}$ with
  $X'\in\mathfrak e|_{\sigma_j(v)}$ and $X_{\alpha}\in\mathfrak g_{\alpha}|_{\sigma_{j}(v)}$. Clearly we
  have $(R_t)_{\ast}(X')\in\mathfrak v|_{\sigma_j(v)t}$, since $\mathfrak e\subset\mathfrak
  t_{\mathbb C}$ is abelian and $t \in \T$. Let $W\in \mathbb CT_t\mathbb T$ and notice
  (since $(R_t)_{\ast}$ is a Lie algebra homomorphism) that
  \[ \begin{aligned}
    [W_{\sigma_j(v)t},(R_t)_{\ast}X_{\alpha}] & = [(R_t)_{\ast}(W_{\sigma_j(v)}),(R_t)_{\ast}(X_{\alpha})] \\
    & =(R_t)_{\ast}[W_{\sigma_j(v)},X_{\alpha}] \\
    & =\alpha(W)(R_t)_{\ast}(X_{\alpha}).
    \end{aligned} \]

  That is, we have $(R_t)_{\ast}(X_{\alpha})\in \operatorname{span}_\C \{ X_{\alpha} \} =\mathfrak
  g_{\alpha,\sigma_j(v)}=\mathfrak g_{\alpha,\sigma_j(v)t}$.
  Therefore, we proved that
  \[
  (\Psi_{j\ast})_{(v,t)}( \pi_{\ast}(\mathfrak v|_{\sigma_j(v)}) \oplus \mathfrak e|_t ) \subset \mathfrak v_{\sigma_j(v)t}
  \]
  and since $\Psi_{j}$ is a diffeomorphism, by dimensional restrictions, the last inclusion is an
  equality.
\end{proof}

\section{The cohomology of $\mathrm T$-type elliptic structures}
\label{sec:toral_elliptic_cohomology}

In this section, we prove the main result regarding the cohomology associated with elliptic
structures of $T$-type. For this, we need to briefly discuss the notation and some facts about sheaf cohomology.
Regarding sheaf cohomology, we need the Künneth formula, the Mayer-Vietoris sequence, and a result
connecting the usual cohomology spaces as defined in the preliminaries section to the sheaf
theoretical definition.

Let $M$ be any smooth manifold endowed with an elliptic structure $\mathcal V$. For a given open set
$U \subset M$, we denote by $\mathcal S(U)$ the set of smooth functions on $U$ that are annihilated
by $\dext'$. It is easy to verify that this gives rise to a sheaf, called the \emph{sheaf of
solutions of $\mathcal V$}. In the special case in which $\mathcal V$ defines a complex structure on
$M$, the sheaf of solutions is called the \emph{sheaf of holomorphic functions on $M$} and it is
denoted by $\mathcal O$. To each sheaf, there are sheaf cohomology spaces denoted by $H^q(M, \mathcal S)$
for $q = 0, 1, \ldots, n = \operatorname{rank} \mathcal V$. Since elliptic structures are locally
solvable \cite[Section VI.7]{treves_hypo-analytic_1992}, the cohomology spaces discussed in the
preliminaries section of this paper are isomorphic to the sheaf-theoretical objects. For an
introduction to sheaf theory, we refer the reader to \cite{warner_foundations_1983} and for a full
treatise on sheaf theory and several tools, we refer the reader to \cite{bredon_sheaf_1997}.

\begin{thm}\label{thm:KunnethIso}
  Let $G$ be a compact and connected Lie group. Let $\mathfrak v = \mathfrak e \oplus \mathfrak b_{\mathfrak t}$ be a left-invariant elliptic
  structure of T-type on $G$, $\Omega = G / \mathbb T,$ and let $\mathcal W$ be the induced complex
  structure on $\Omega$ given by $\mathcal W = \pi_*(\mathfrak v)$.
   
  Then, for any $q$ there are isomorphisms $$H^{q}(G,\mathfrak v) \cong \bigoplus_{\nu+\mu=q}H^{\nu}(\mathbb
  T,\mathfrak e) \otimes H^{\mu}(\Omega,\mathcal W).$$
\end{thm}

\begin{proof}
  It follows from Proposition \ref{prop:structurebundle} that the local trivializations $\Phi_j$ are
  compatible with the elliptic structure $\mathfrak e$ and the elliptic structure obtained from the
  direct sum $\mathcal W \oplus \mathfrak e$, exactly the setting of the Leray-Hirsch argument used
  in \cite{jacobowitz_levi-flat_2023}. In fact, elliptic structures are locally solvable, which
  allows us to use sheaf-theoretic tools, and the cohomology spaces $H^\nu(\mathbb T,\mathfrak e)$ are
  finite-dimensional for all $\nu$, since $\T$ is compact and $\mathfrak e$ is elliptic. Thus, we can 
  compute the cohomology locally via the Künneth formula and glue it via Mayer-Vietoris and the Five
  Lemma. We refer the reader to \cite{jacobowitz_levi-flat_2023} for the details. The argument there
  applies in the same way.
\end{proof}

Now, we are finally ready to prove Theorem \ref{thm:reduction_to_torus}.
\begin{proof}[Proof of Theorem \ref{thm:reduction_to_torus}] From Theorem \ref{thm:KunnethIso}, we
have the isomorphism $$H^{q}(G,\mathfrak v) \cong \bigoplus_{\nu+\mu=q}H^{\nu}(\mathbb T,\mathfrak e)
\otimes H^{\mu}(\Omega,\mathcal W).$$
  From \cite[Chapter 8]{besse_einstein_2008}, $(\Omega,\mathcal W)$ has positive first Chern
  class, so we can apply  \cite[Corollary 11.25]{besse_einstein_2008} to obtain
  $H^{\mu}(\Omega,\mathcal W)=0$ for all $\mu>0$ and $H^{0}(\Omega,\mathcal W)=\mathbb C$. We conclude
  that
  \[
    H^{q}(G,\mathfrak v) \cong H^{q}(\mathbb T,\mathfrak e) \otimes H^{0}(\Omega,\mathcal W)
    \cong H^{q}(\mathbb T,\mathfrak e).
  \]
  
  To conclude, we recall that all elements in the range of \eqref{eq:torus_inclusion} are classes with left-invariant representatives, and so, it is contained in the range of
\eqref{eq:the_homomorphism}. Since $H^q(G, \mathfrak v)$ and $H^{q}(\mathbb T, \mathfrak e)$ are finite-dimensional and have the same dimension, the surjectivity of \eqref{eq:torus_inclusion} follows from linear algebra, and we conclude that \eqref{eq:the_homomorphism} is surjective.
\end{proof}

\section{Elliptic structures on the torus}
\label{sec:elliptic_torus}

In this section, we briefly show how to compute the cohomology of left-invariant elliptic structures on a torus. The formulas presented here are adaptations of ideas from \cite{bergamasco_global_1999}. Detailed computations, in the CR setting, can be seen in \cite{jacobowitz_levi-flat_2023}.

Let $f \in \mathscr C^\infty(\T^N)$. The Fourier coefficients of $f$ are given by
\[
  \widehat f (\xi) \doteq \frac{1}{(2\pi)^N} \int f(x) e^{-i\xi x} dx, \qquad \xi \in \Z^N.
\]

The sequence $\{\widehat f (\xi)\}$ is a rapidly decreasing sequence, that is, for all integers $L > 0$ there exists $C_L$ such that
\[
  |\widehat f (\xi)| \leq \frac{C_L}{(1 + | \xi |)^L}, \qquad \xi \in \Z^N.
\]
with $|\xi|^2 = \sum_{j=1}^N \xi_j^2$.

The following is also easy to verify. Given a rapidly decreasing sequence of complex numbers $\{a_\xi\}_{\xi \in \Z^N}$, we have that $f(x) \doteq \sum_{\xi \in \Z^N} a_\xi e^{i \xi x}$ converges in the $\mathscr C^\infty(\T^N)$-topology to a smooth function such that $\widehat f(\xi) = a_\xi$ for all $\xi \in \Z^N$.

Let $\mathfrak e$ be an elliptic structure on a torus $\mathbb T^N$ of dimension $n \doteq \dim \mathfrak
e \geq \lceil \frac{N}{2} \rceil$\footnote{The symbol $ \lceil x \rceil$, $x \in \R$, denotes the smallest integer greater than or equal to $x$. } and let $m=N-n$. Since the torus is abelian, the structure is bi-invariant. Let $\{L_1,\dots,L_n\}$ be a basis for $\mathfrak e$ and
complete it to a basis for $\mathfrak t^N_{\C}$, denoted $\{L_1,\dots,L_n,L_{n+1},\dots, L_N\}$. Let
$\{\tau_1,\dots,\tau_N\}$ denotes the dual basis.
Without loss of generality, we may assume that $L_{j+n}=\overline{L_j}$ for $j=1,\dots, m$ and that
$L_{i}$ are real for $i \in \{m+1,\dots,n\}$. We can write the basis elements
$$L_j=\sum^N_{l=1}a_{jl}\frac{\partial}{\partial x_l},\quad j=1,\dots,N$$ with
$a_{j,l}=\overline{a_{j+n,l}}\in\mathbb C$ for $j \in \{1,\dots,m\}$, and $a_{j,l}\in\mathbb R$ for
$j \in \{m+1,\dots, n\}$. We write the symbol of $L_j$ as $$\widehat{L_j}(\xi)=i\sum^N_{l=1}a_{jl}\xi_l,\quad \xi\in\mathbb Z^N.$$ Since the Lie algebra $\mathfrak e \subset \mathfrak t_{\C}$ defines an
elliptic structure, for each $\xi \in \Z^N \setminus \{0\}$, at least one of the symbols $\widehat{L_j}(\xi)$ is invertible and so, for each $\xi \in \Z^N \setminus \{0\}$, there is $\sigma \in \{1,\dots, n\}$ such that
\[
  |\widehat{L_{\sigma}}(\xi)|=\max\{|\widehat{L_j}(\xi)|:j=1,...,n\} > 0.
\]

Furthermore, smooth forms $u \in \mathscr C^{\infty}(\mathbb T^N,\UL^q)$ can be written as
\[
  u = \psum_{|I|=q} u_I\tau_I, \footnote{The symbol $'$ on the sum indicates that the sum is over increasing multi-indices $I$.}
\]
with $I=(i_1, \dots, i_q)$ such that $1 \leq i_1 < i_2 < \cdots < i_q \leq n$, $\tau_I = \tau_{i_1} \wedge \dots \wedge \tau_{i_q}$, and $u_I \in \mathscr C^{\infty}(\mathbb T^N)$. The operator $\dext'$ acts on a smooth $q$-form $u$ as
follows $$\dext'u = \psum_{|I| = q} \sum_{j=1}^n (L_ju_I) \tau_j \wedge \tau_I.$$

Notice that, with a basis for $\mathfrak t^N_\C$ fixed, we can define the Fourier series of a form by defining it coefficient-wise. Since for each $q \in \{0, \ldots, n\}$ we can identify $\mathscr C^\infty(\T^N, \UL^q)$ with $\mathscr C^\infty(\T^N) \otimes \UL^q$ and since $\UL^q$ is finite-dimensional, the extension of the Fourier series from $\mathscr C^\infty(\T^N)$ to $\mathscr C^\infty(\T^N, \UL^q)$ is independent of the choice of basis. A choice of coordinates, however, makes computations easy. For example, it is trivial to verify that $\dext'$ acting on a form $\widehat u(\xi)$ is given by
\[
  \widehat {\dext' u}(\xi) = \psum_{|I| = q} \sum_{j=1}^n (\widehat L_j (\xi) \widehat{u_I}(\xi)) \tau_j \wedge \tau_I.
\]

The following lemma gives us a frequency-wise condition for a form to be $\dext'$-closed.

\begin{lem}
  A $q$-form $u\in\mathscr C^{\infty}(\mathbb T^N,\UL^q)$ is a $\dext'$-closed form if and only if
  $$\biggl(\sum_{j=1}^n\widehat{L_j}(\xi)\tau_j\biggr)\wedge\widehat u(\xi)=0,\quad
  \forall\xi\in\mathbb Z^N,$$ where $\widehat u(\xi)=\sum_{|I|=q}\widehat{u_I}(\xi)\tau_I$.
\end{lem}

\begin{proof}
  Omitted. See \cite{jacobowitz_levi-flat_2023} for details.
\end{proof}

Since $\T^N$ is compact and $\dext'$ forms an elliptic complex, we know that $H^q(\T^N, \mathfrak e)$ is finite dimensional and that $\dext' : \mathscr C^\infty(\T^N, \UL^{q-1}) \to  \mathscr C^\infty(\T^N, \UL^{q})$ has closed range for all $q$ \cite[Chapter IV.5]{wells_differential_2008}.

The following lemma gives us a frequency-wise left-inverse for $\dext'$. For $\xi \in \mathbb Z^N \setminus \{0\}$, and
  $\sigma\in\{1,\dots,n\}$ we use the notation
  $$|\widehat{L_{\sigma}}(\xi)|=\max\{|\widehat{L_j}(\xi)|:j=1,...,n\}.$$

\begin{lem}
\label{lem:dprime_inv}
  Given $u\in\mathscr C^{\infty}(\mathbb T^N,\UL^q)$ a $\dext'$-closed $q$-form with $q \geq 1$. For each $\xi \in \mathbb Z^N \setminus \{0\}$, the $(q-1)-$form
  $$\widehat v(\xi) \doteq \sum_{|J|=q;\sigma\in
  J}\varepsilon_{\sigma,J}\footnote{The symbol $\varepsilon_{\sigma,J}$ denotes the signature of the permutation $(\sigma, J\backslash\sigma)$.}\frac{1}{\widehat{L_{\sigma}}(\xi)}u_J(\xi)\tau_{J\setminus\sigma}$$
  satisfies $$\sum^n_{j=1}\widehat{L_j}(\xi)\tau_j\wedge\widehat v(\xi)=\widehat u(\xi).$$
\end{lem}

\begin{proof}
  Omitted. See \cite{jacobowitz_levi-flat_2023} for details.
\end{proof}

\begin{lem}
\label{lem:dprime_closed_exact}
  Given a $\dext'$-closed $q$-form $u$, the form
  \[
    u^* = \sum_{\xi\in\mathbb Z^N\setminus\{0\}} e^{i\xi x}\widehat u(\xi)
  \] is $\dext'$-exact.
\end{lem}

\begin{proof}
  For all $\xi \in \mathbb Z^N \setminus\{0\}$, by Lemma \ref{lem:dprime_inv}, there exists $\widehat v(\xi)$ such that
  \[
    \sum^n_{j=1}\widehat{L_j}(\xi)\tau_j \wedge \widehat v(\xi) =\widehat u(\xi)
  \] so that we can define
  \[
    v_M = \sum_{ \xi \in \mathbb Z^N \setminus\{0\}, |\xi| < M} \widehat v (\xi) e^{ix\xi}.
  \]
    
  Notice that $v_M$ is a smooth form, since it is a finite sum.
    
  Let
  \[
    u^*_M = \sum_{\xi \in \mathbb Z^N \setminus\{0\}, |\xi| < M} \widehat u(\xi) e^{i\xi x} 
  \]
  and notice that it satisfies $\dext'v_M = u^*_M$, and so $u^*_M$ is in the range of $\dext'_{q-1}$ for all $M$. Since $\dext'_{q-1}$ has closed range, and we have that $u^*_M \to u^*$ in the $\mathscr C^\infty$-topology, we have that $u^*$ is in the range of $\dext'_{q-1}$.
\end{proof}

\begin{prp}
  Let $\mathfrak e$ be an elliptic structure of dimension $n$ on the torus $\mathbb T^N$. Then we
  have $$\begin{cases}
    \dim H^q(\mathbb T^N,\mathfrak e) = \binom{n}{q}\ & \text{ if }\ 0 \leq q \leq n,\\
    \dim H^q(\mathbb T^N,\mathfrak e)=0 & \text{ if } \ q>n.
  \end{cases}$$
\end{prp}

\begin{proof}
  It follows from Lemma \ref{lem:dprime_closed_exact} that all cohomology classes of $H^q(\mathbb T^N,\mathfrak e)$ have left-invariant representatives, and so the $\dext'$-cohomology can be computed using only bi-invariant forms. This means that $H^q(\mathbb T^N,\mathfrak e)$ is isomorphic to the Chevalley-Eilenberg complex of $\mathfrak e$. Since $\mathfrak e$ is abelian, we easily verify $H^q(\mathbb T^N,\mathfrak e) = \Wedge^q \mathfrak e^*$ which concludes the proof.
\end{proof}

\section*{Acknowledgements}

  The first author is grateful to Prof. George Marinescu for his hospitality during the first author's research stay at the University of Cologne, where the majority of this work was carried out. The stay was funded by the Marco Polo programme of the University of Bologna. The second author was funded by the Deutsche Forschungsgemeinschaft (DFG) under grant JA~3453/2-1.

\bibliographystyle{amsalpha}
\bibliography{references}

\end{document}